\documentclass[12pt]{amsart}
\usepackage{amsmath, amssymb, amsfonts, amsthm, fullpage, stmaryrd,  url, hyperref}
\usepackage[all]{xy}

\newtheorem{theorem}{Theorem}[section]
\newtheorem{lemma}[theorem]{Lemma}
\newtheorem{cor}[theorem]{Corollary}
\newtheorem{prop}[theorem]{Proposition}

\theoremstyle{definition}
\newtheorem{defn}[theorem]{Definition}
\newtheorem{example}[theorem]{Example}

\newtheorem{hypothesis}[theorem]{Hypothesis}

\newtheorem{remark}[theorem]{Remark}

\numberwithin{equation}{theorem}

\newcommand{\bB}{\mathbf{B}}

\newcommand{\bv}{\mathbf{v}}

\newcommand{\CC}{\mathbb{C}}

\newcommand{\Qp}{\mathbb{Q}_p}
\newcommand{\QQ}{\mathbb{Q}}
\newcommand{\RR}{\mathbb{R}}

\newcommand{\ZZ}{\mathbb{Z}}
\newcommand{\calF}{\mathcal{F}}
\newcommand{\calH}{\mathcal{H}}

\newcommand{\calO}{\mathcal{O}}
\newcommand{\calR}{\mathcal{R}}

\newcommand{\frakm}{\mathfrak{m}}
\newcommand{\frako}{\mathfrak{o}}

\newcommand{\dual}{\vee}

\DeclareMathOperator{\bd}{bd}
\DeclareMathOperator{\BPair}{\mathbf{BPair}}
\DeclareMathOperator{\Dfm}{\mathbf{Dfm}}
\DeclareMathOperator{\dR}{dR}

\DeclareMathOperator{\Gal}{Gal}

\DeclareMathOperator{\Hom}{Hom}
\DeclareMathOperator{\id}{id}
\DeclareMathOperator{\Ind}{Ind}
\DeclareMathOperator{\Iw}{Iw}
\DeclareMathOperator{\Maxspec}{Maxspec}
\DeclareMathOperator{\Cont}{Cont}

\DeclareMathOperator{\PhiGamma}{\mathbf{\Phi \Gamma}}
\DeclareMathOperator{\PhiGammabreve}{\breve{\mathbf{\Phi \Gamma}}}
\DeclareMathOperator{\PhiGammatilde}{\widetilde{\mathbf{\Phi \Gamma}}}
\DeclareMathOperator{\PhiMod}{\mathbf{\Phi Mod}}
\DeclareMathOperator{\Pic}{Pic}
\DeclareMathOperator{\Proj}{Proj}
\DeclareMathOperator{\rank}{rank}
\DeclareMathOperator{\Rep}{\mathbf{Rep}}
\DeclareMathOperator{\Res}{Res}
\DeclareMathOperator{\Spec}{Spec}
\DeclareMathOperator{\Trace}{Trace}
\DeclareMathOperator{\triv}{triv}
\DeclareMathOperator{\VB}{\mathbf{VB}}

\begin{document}

\title{On categories of $(\varphi, \Gamma)$-modules}
\author{Kiran S. Kedlaya and Jonathan Pottharst}
\thanks{Kedlaya was supported by NSF (grant DMS-1501214), UC San Diego (Warschawski Professorship), Guggenheim Fellowship (fall 2015). Thanks to Ofer Gabber for helpful feedback.}
\date{February 26, 2017}

\begin{abstract}
Let $K$ be a complete discretely valued field of mixed characteristics $(0,p)$ with perfect residue field. One of the central objects of study in $p$-adic Hodge theory is the cate\-gory of continuous representations of the absolute Galois group of $K$ on finite-dimensional $\QQ_p$-vector spaces. In recent years, it has become clear that this category can be studied more effectively by embedding it into the larger category of $(\varphi, \Gamma)$-modules; this larger category plays a role analogous to that played by the category of vector bundles on a compact Riemann surface in the Narasimhan-Seshadri theorem on unitary representations of the fundamental group of said surface. This category turns out to have a number of distinct natural descriptions, which on one hand suggests the naturality of the construction, but on the other hand forces one to use different descriptions for different applications. We provide several of these descriptions and indicate how to translate certain key constructions, which were originally given in the context of modules over power series rings, to the more modern context of perfectoid algebras and spaces.
\end{abstract}

\maketitle

Throughout, let $p$ be a prime number and let $K$ be a \emph{$p$-adic field}, by which we mean a complete discretely valued field of mixed characteristics $(0,p)$ and perfect residue field. For instance, $K$ may be a finite extension of $\QQ_p$, but we do not restrict to this case unless otherwise specified.

In $p$-adic Hodge theory, one studies the relationship between different cohomology theories associated to algebraic (and more recently analytic) varieties over $K$. For example, by analogy with the comparison between the Betti and de Rham cohomologies associated to a complex algebraic variety, one has a \emph{$p$-adic comparison isomorphism} comparing the $p$-adic \'etale and algebraic de Rham cohomologies of a variety over $K$. (This result has a long, rich, and continuing history which we do not wish to review here; see the introductions of \cite{scholze2} and \cite{bhatt-morrow-scholze} for up-to-date synopses.) 

Continuing with this analogy, just as one encapsulates the Betti--de Rham comparison isomorphism in the construction of a \emph{Hodge structure} associated to a complex algebraic variety, one would like to encode the $p$-adic comparison isomorphism into an object associated to a $K$-variety that ``remembers'' certain cohomology groups and their associated structures. One reason to do this is to study situations where one has putative instances of the comparison isomorphism corresponding to varieties which are expected to exist but not yet constructed; this situation occurs naturally in the study of Shimura varieties \cite{liu-zhu}.

A first approximation to such a package is Fontaine's definition of a
\emph{$(\varphi, \Gamma)$-module} \cite{fontaine-phigamma}. To a continuous representation of $G_K$ on a finite-dimensional $\QQ_p$-vector space, Fontaine associates a module over a certain commutative ring (described explicitly in terms of formal Laurent series; see \S\ref{sec:categories}) equipped with semilinear actions of certain endomorphisms of the base ring (the eponymous $\varphi$ and
$\Gamma$). The fact that the representation can be recovered from this module rests on two pivotal observations: first, one can describe representations of Galois groups of fields of characteristic $p$ on finitely generated $\ZZ_p$-modules in a manner reminiscent of Artin--Schreier theory,
as shown by Katz
\cite[4.1.1]{katz-modular};
second, the infinite cyclotomic extension $K(\mu_{p^\infty})$ has the same Galois group as the field of Laurent series over the residue field of $K(\mu_{p^\infty})$,
as shown by Fontaine--Wintenberger \cite{fontaine-wintenberger}.

Given a smooth proper $K$-variety, its $p$-adic \'etale cohomology admits a continuous $G_K$-action and thus may be fed into Fontaine's $(\varphi, \Gamma)$-module functor. Fontaine had already established
how to pass functorially from $p$-adic \'etale cohomology to de Rham cohomology, so in principle the same information is captured in the $(\varphi, \Gamma)$-module; however, there is no obvious way to convert the $(\varphi, \Gamma)$-module into de Rham cohomology without first passing back to \'etale cohomology.

This defect was subsequently remedied by work of Cherbonnier--Colmez \cite{cherbonnier-colmez} and Berger \cite{berger}, which shows that Fontaine's construction may be modified to use a different base ring in which formal Laurent series are replaced by rigid analytic functions on suitable annuli (again see \S\ref{sec:categories} for precise definitions).
This has the crucial advantage of making ``evaluation at a point'' make enough sense so that Fontaine's \'etale--de Rham construction factors through the $(\varphi, \Gamma)$-module. Among the various applications of this refinement, we single out one which will be relevant later in our story: one can use $(\varphi, \Gamma)$-modules in this sense to give explicit formulas in Iwasawa theory, such as for the Bloch--Kato exponential function and the Perrin-Riou reciprocity map \cite{berger-explicit}.

Here we arrive at the point of departure for this paper: the Iwasawa theory we just alluded to is nowadays retronymically called \emph{cyclotomic} Iwasawa theory, because similar considerations have subsequently been applied to towers of number fields whose Galois groups are various (not necessarily commutative) $p$-adic Lie groups.
However, the construction of $(\varphi, \Gamma)$-modules described above gives a privileged role to the cyclotomic towers, so it is not immediately evident that they can be used to study Iwasawa theory beyond the cyclotomic case. While some initial work in this direction does exist \cite{berger-fourquaux, berger-schneider-xie, schneider-venjakob}, our aim here is not to address this issue directly, but to instead work at a more foundational level: we describe $(\varphi, \Gamma)$-modules, and the constructions used in Berger's explicit formulas, in terms that do not distinguish the cyclotomic tower from other $p$-adic Lie towers.

Before saying more about this, it is important to emphasize the fact that the functor from 
$p$-adic representations of $G_K$ to $(\varphi, \Gamma)$-modules in the sense of Berger is only a full embedding, not an equivalence of categories.
Using work of the first author \cite{kedlaya-annals}, one may characterize the essential image of the functor (the so-called \emph{\'etale $(\varphi, \Gamma)$-modules}) as the \emph{semistable} objects for a suitable degree function; this immediately suggests an analogy with the Narasimhan--Seshadri theorem \cite{narasimhan-seshadri} identifying unitary representations of the fundamental group of a compact Riemann surface with polystable bundles.
The computation of (continuous) Galois cohomology of representations in terms of $(\varphi, \Gamma)$-modules, worked out in Fontaine's setting by Herr \cite{herr, herr-tate},
adapts to Berger's setting and extends to non-\'etale objects by work of R. Liu \cite{liu-herr}. These developments have led to several new applications of $p$-adic Hodge theory, such as the construction of a $p$-adic local Langlands correspondence for $\mathrm{GL}_2(\mathbb{Q}_p)$ with good interpolation properties
\cite{colmez-langlands} and the study of the geometry of eigenvarieties via triangulations
of $(\varphi, \Gamma)$-modules \cite{kpx, bergdall, breuil-hellmann-schraen, breuil-hellmann-schraen2}; this paper may have some relevance to such applications also, but we refrain from speculating on this point here.

A first step towards realizing our goal of getting the cyclotomic tower out of the theory of $(\varphi, \Gamma)$-modules is to describe the category in a more agnostic fashion.
We introduce two of these: one by Berger in the language of \emph{$B$-pairs} and one by Fargues--Fontaine in terms of vector bundles on certain one-dimensional noetherian schemes
(sometimes called \emph{Fargues--Fontaine curves}). 
The latter description arises very naturally within the geometric reinterpretation of $p$-adic Hodge theory in the language of \emph{perfectoid spaces}, as in the work of Scholze \cite{scholze1, scholze2, scholze-icm} and Kedlaya--Liu \cite{kedlaya-liu1, kedlaya-liu2}.

However, these descriptions alone do not suffice to expunge the cyclotomic tower from
the theory of $(\varphi, \Gamma)$-modules from the point of view of applications to Iwasawa theory. This is because Berger's explicit formulas depend crucially on a certain construction
involving reduced traces in the cyclotomic tower, which ultimately manifests as an operator on the power series used in the definition of the base ring of $(\varphi, \Gamma)$-modules (the $\psi$ map; see Definition~\ref{D:psi}). To eliminate this construction, we make crucial use of 
the theory of \emph{arithmetic deformations} of $(\varphi, \Gamma)$-modules, as developed with L. Xiao in \cite{kpx}; this allows us to exchange the explicit use of the cyclotomic tower in the classical construction of $(\varphi, \Gamma)$-modules for an arithmetic deformation parametrizing cyclotomic twists.
This makes it natural to consider other deformations corresponding to other $p$-adic Lie groups, including nonabelian ones. A natural next step would be to try to interpret results from any of \cite{berger-fourquaux, berger-schneider-xie, schneider-venjakob} in this framework, but we stop short of this point; 
see Remark~\ref{R:berger explicit} and \S\ref{sec:coda} for a brief discussion.

Before concluding this introduction, we set a few running notations. Our primary model for these and other notations is \cite{kpx}.
\setcounter{theorem}{0}
\begin{hypothesis}
Throughout this paper, as in this introduction, let $K$ be a complete discretely valued field of mixed characteristics with perfect residue field $k$ and absolute Galois group $G_K$; we do not assume $k$ is finite (i.e., that $K$ is a finite extension of $\QQ_p$) unless explicitly specified.
Put $F = W(k)[1/p]$ for $W(k)$ the ring of Witt vectors over $k$, so that $K/F$ is a finite totally ramified extension.
Let $A$ be an affinoid algebra over $\QQ_p$ in the sense of Tate, rather than the more expansive sense of Berkovich. 
\end{hypothesis}

\section{The original category of $(\varphi, \Gamma)$-modules}
\label{sec:categories}

We begin by describing the original construction of the category of $(\varphi, \Gamma)$-modules,
into which the category of continuous representations of $G_K$ on finite-dimensional $\QQ_p$-vector spaces embeds; this is most explicit in the case $K=F$.
In preparation for our later discussion, we escalate the level of generality to accommodate representations valued in affinoid algebras.

\begin{defn}
Let $\Rep_A(G_K)$ denote the category of continuous representations of $G_K$ on finite projective $A$-modules. With more work, it would be possible to consider also representations on arbitrary finitely generated $A$-modules; we will not attempt this here.
\end{defn}

\begin{defn}
Let $\calR^\infty_{F,A}$ be the ring of rigid analytic functions on the disc $\left| \pi \right| < 1$ over $F \widehat{\otimes}_{\QQ_p} A$. This ring is complete for the topology of uniform convergence on quasicompact subspaces (Fr\'echet topology). The ring admits a continuous endomorphism $\varphi$ defined by the formula
\begin{equation} \label{eq:phi formula}
\varphi \left( \sum_n c_n \pi^n \right) =  \sum_n \varphi_F(c_n) ((1 + \pi)^p-1)^n,
\end{equation}
where $\varphi_F$ denotes the $A$-linear extension of Witt vector Frobenius map on $F$.
The group $\Gamma = \ZZ_p^\times$ also admits a continuous action on $\calR^\infty_{F,A}$ defined by the formula
\begin{equation} \label{eq:gamma formula}
\gamma \left( \sum_n c_n \pi^n \right) = \sum_n c_n ((1 + \pi)^\gamma-1)^n \qquad (\gamma \in \Gamma)
\end{equation}
under the interpretation of $(1 + \pi)^\gamma$ as the binomial series 
\[
(1 + \pi)^\gamma = \sum_{n=0}^\infty \frac{\gamma(\gamma-1)\cdots (\gamma-n+1)}{n!} \pi^n.
\]
Note that the actions of $\varphi$ and $\Gamma$ commute.
\end{defn}

\begin{defn}
Let $\calR_{F,A}$ be the direct limit of the rings of rigid analytic functions on the annuli $* < \left| \pi \right| < 1$ over $F \widehat{\otimes}_{\QQ_p} A$. This ring is complete for the direct limit topology (in the category of locally convex topological $\QQ_p$-vector spaces) induced by the topologies of uniform convergence on quasicompact subspaces (the \emph{LF topology}). We extend the actions of $\varphi$ and $\Gamma$ on $\calR^\infty_{F,A}$ to continuous actions on $\calR_{F,A}$ using the same formulas \eqref{eq:phi formula}, \eqref{eq:gamma formula}.
Note that $\calR_{F,A}$ is connected whenever $A$ is.
\end{defn}

\begin{defn}
A \emph{$(\varphi, \Gamma)$-module} over $\calR_{F,A}$ is a finite projective $\calR_{F,A}$-module $M$ equipped with commuting semilinear actions of $\varphi$ and $\Gamma$ such that the action of $\Gamma$ is continuous for the LF topology. Here by a \emph{semilinear action} of $\varphi$ on a module $M$, we mean a semilinear map $\varphi_M: M \to M$ induced by an isomorphism $\varphi^* M \to M$; note that the isomorphism property does not ensure that $\varphi_M$ acts bijectively on $M$, as this already fails for $M = \calR_{F,A}$ with the standard $\varphi$-action.

Let $\PhiGamma_{F,A}$ denote the category of $(\varphi, \Gamma)$-modules over $\calR_{F,A}$, viewed as an exact tensor category with rank function $\rank_F: \Phi\Gamma_{F,A} \to \Cont(\Spec A,\ZZ)$ computing the rank of the underlying $\calR_{F,A}$-module.
\end{defn}

We will establish the following result in \S\ref{sec:alternate}.
\begin{theorem} \label{T:phi gamma embedding1}
There exists a full embedding $\Rep_A(G_F) \to \PhiGamma_{F,A}$.
\end{theorem}
In the interim, let us see how this result can be used to define a corresponding category with $F$ replaced by $K$.

\begin{defn}
Let $\calR_{K,A} \in \PhiGamma_{F,A}$ be the object of rank $[K:F]$ corresponding to $\Ind^{G_F}_{G_K} \rho_{\triv}$ via Theorem~\ref{T:phi gamma embedding}. The canonical isomorphisms $\rho_{\triv} \otimes \rho_{\triv} \cong \rho_{\triv}^\dual \otimes \rho_{\triv} \cong \rho_{\triv}$
then correspond to an associative morphism $\mu_K: \calR_{K,A} \otimes \calR_{K,A} \to \calR_{K,A}$; this gives $\calR_{K,A}$ the structure of a finite flat $\calR_{F,A}$-algebra equipped with continuous actions of $\varphi$ and $\Gamma$.

Let $\PhiGamma_{K,A}$ be the category of pairs $(M, \mu)$ for which $M \in \PhiGamma_{F,A}$
and $\mu: \calR_{K,A} \otimes_{\calR_{F,A}} M \to M$ is a morphism which is associative with respect to $\mu_K$, i.e., the compositions
\begin{gather*}
\calR_{K,A} \otimes \calR_{K,A}  \otimes M \stackrel{\mu_K \otimes 1}{\to} M_K \otimes M \stackrel{\mu}{\to} M, \\
\calR_{K,A} \otimes \calR_{K,A} \otimes M \stackrel{1 \otimes \mu}{\to} \calR_{K,A} \otimes  M \stackrel{\mu}{\to} M
\end{gather*}
coincide. 
In other words, these are finite projective $\calR_{K,A}$-modules equipped with commuting semilinear continuous actions of $\varphi$ and $\Gamma$.
We again view $\PhiGamma_{K,A}$ as an exact tensor category with rank function $\rank_K = \rank_F / [K:F]$ computing the rank of the underlying $\calR_{K,A}$-module.

Let $K'$ be a finite extension of $K$. Define the induction functor
$\Ind: \PhiGamma_{K',A} \to \PhiGamma_{K,A}$ 
and the restriction functor $\Res: \PhiGamma_{K,A} \to \PhiGamma_{K',A}$
by restriction of scalars and extension of scalars, respectively, along the natural map $\calR_{K,A}  \to \calR_{K',A}$.
\end{defn}

We may then formally promote Theorem~\ref{T:phi gamma embedding1} as follows.
\begin{theorem} \label{T:phi gamma embedding}
There exists a full embedding $\Rep_A(G_K) \to \PhiGamma_{K,A}$ compatible with induction and restriction on both sides.
\end{theorem}

\begin{defn}
The category $\PhiGamma_{K,A}$ admits duals, and hence internal Homs: the dual of $M \in \PhiGamma_{K,A}$ is the module-theoretic dual $M^\dual = \Hom_{\calR_{K,A}}(M, \calR_{K,A})$
with the actions of $\varphi, \Gamma$ constructed so that the canonical $\calR_{K,A}$-linear morphism $M^\dual \otimes_{\calR_{K,A}} M \to \calR_{K,A}$ is a morphism in $\PhiGamma_{K,A}$. (Note that the definition of the $\varphi$-action on $M^\dual$ depends on the fact that the action of $\varphi$ corresponds to an isomorphism $\varphi^*M \to M$, not just an arbitrary $\calR_{K,A}$-linear morphism.)
For $M$ corresponding to $V \in \Rep_A(G_K)$ via Theorem~\ref{T:phi gamma embedding}, $M^\dual$ corresponds to the contragredient representation $V^\dual$.

Let $\calR_{K, A}(1)$ denote the object of $\PhiGamma_{K, A}$ corresponding to the cyclotomic character $\chi$ in $\Rep_A(G_K)$ via Theorem~\ref{T:phi gamma embedding}.
Concretely, $\calR_{K,A}(1)$ can be written as the free module of rank 1 on a generator $\varepsilon$ satisfying
\[
\varphi(\varepsilon) = \varepsilon, \qquad \gamma(\varepsilon) = \chi(\gamma) \varepsilon \qquad (\gamma \in \Gamma).
\]
For $M \in \PhiGamma_{K,A}$, define the \emph{Cartier dual}
\[
M^* = M^\dual(1) = M^\dual \otimes_{\calR_{K,A}} \calR_{K,A}(1) \in \PhiGamma_{K,A};
\]
for $M$ corresponding to $V \in \Rep_A(G_K)$ via Theorem~\ref{T:phi gamma embedding}, $M^*$ corresponds to the Cartier dual of $V$ (i.e., the contragredient of $V$ twisted by the cyclotomic character).
\end{defn}

\begin{remark} \label{R:not connected}
The description of $\PhiGamma_{K,A}$ given above is consistent with \cite{kedlaya-new-phigamma} but not with most older references. The reason is that even if $\Maxspec A$ is connected, in general $\Maxspec \calR_{K,A}$ is not connected; it is more typical to replace it with one of its connected components, and to replace $\Gamma$ with the stabilizer of that component.
See Remark~\ref{R:not connected2} and \cite[Remark~2.2.12]{kedlaya-new-phigamma} for further discussion.
\end{remark}

\begin{remark}
The base ring in Fontaine's original theory of $(\varphi, \Gamma)$-modules was not the ring $\calR_{K, \QQ_p}$, but rather the completion of the subring of elements of $\calR_{K, \QQ_p}$
which are bounded (meaning equivalently that their coefficients or their values are bounded).
This ring cannot naturally be interpreted in terms of functions on a rigid analytic space.
\end{remark}

\section{Interlude on perfectoid fields}
\label{sec:perfectoid}

In preparation for giving alternate descriptions of the category $\PhiGamma_{K,A}$, we introduce the basic theory of \emph{perfectoid fields}, which subsumes the earlier theory of \emph{norm fields} on which the classical theory of $(\varphi, \Gamma)$-modules is built; 
we briefly discuss the relationship with the older theory in 
Remark~\ref{R:deeply ramified}, deferring to \cite{kedlaya-new-phigamma} for more historical discussion. In the process, we must do a bit of extra work in order to accommodate the coefficient ring $A$.

\begin{defn}
Let $L$ be a field containing $K$ which is complete with respect to a nonarchimedean absolute value, denoted $\left| \cdot \right|$. Let $\frako_L$ denote the valuation subring of $L$ (i.e., elements of norm at most $1$). We say $L$ is \emph{perfectoid} if $L$ is not discretely valued and the Frobenius map on $\frako_L/(p)$ is surjective.
\end{defn}

\begin{example}  \label{exa:cyclotomic}
Suppose that $K = F$ and let $L$ be the completion of $K(\mu_{p^\infty})$.
Then 
\begin{align*}
\frako_L &\cong (W(k)[\zeta_p, \zeta_{p^2}, \dots]/(1 + \zeta_p + \cdots + \zeta_p^{p-1}, \zeta_p - \zeta_{p^2}^p, \zeta_{p^2} - \zeta_{p^3}^p, \dots))^{\wedge}_{(p)} \\
\frako_L/(p) &\cong k[T_1, T_2, \dots]/(1+T_1 + \cdots + T_1^{p-1}, T_1 - T_2^p, T_2 - T_3^p, \dots),
\end{align*}
so $L$ is perfectoid; the same will hold for general $K$ by Theorem~\ref{T:perfectoid} below.
For some more general results that subsume this example, see Remark~\ref{R:deeply ramified} and Lemma~\ref{L:deeply ramified is perfectoid}.
\end{example}

\begin{hypothesis}
For the remainder of \S\ref{sec:perfectoid}, let $L$ be a perfectoid field.
\end{hypothesis}

\begin{theorem} \label{T:perfectoid}
Define the multiplicative monoids
\[
\frako_{L^{\flat}} = \varprojlim_{x \mapsto x^p} \frako_L, \qquad
L^{\flat} = \varprojlim_{x \mapsto x^p} L.
\]
\begin{enumerate}
\item[(a)]
There is a unique way to promote $\frako_{L^{\flat}}$ and $L^{\flat}$ to rings
in such a way that the map $\frako_{L^{\flat}} \to L^{\flat}$ and the composition $\frako_{L^{\flat}} \to \frako_L \to \frako_L/(p)$ become ring homomorphisms.
(The map $\frako_{L^{\flat}} \to \frako_L$ is multiplicative but not additive.)
\item[(b)]
The ring $L^{\flat}$ is a perfect field. In addition,
the function $L^{\flat} \to L \stackrel{\left| \cdot \right|}{\to} \RR$ is an absolute value with respect to which $L^{\flat}$ is complete with valuation subring $\frako_{L^{\flat}}$.

\item[(c)]
The field $L^\flat$ is also perfectoid.
\item[(d)]
Any finite extension of $L$, equipped with the unique extension of the absolute value, is again perfectoid.
\item[(e)]
The functor $L' \mapsto L^{\prime \flat}$ defines an equivalence of categories between finite extensions of $L$ and $L^{\flat}$, and thereby a canonical isomorphism $G_L \cong G_{L^{\flat}}$.
\end{enumerate}
\end{theorem}
\begin{proof}
See \cite[\S 1]{kedlaya-new-phigamma} and references therein.
\end{proof}

\begin{defn}
Define the field $L^\flat$ and equip it with an absolute value as per
Theorem~\ref{T:perfectoid}(b). 
For $r>0$, let $W^r(L^\flat)$ be the set of $x = \sum_{n=0}^\infty p^n [\overline{x}_n] \in W(L^\flat)$ such that $p^{-n} \left| \overline{x}_n \right|^r \to 0$ as $n \to \infty$.
By \cite[Proposition~5.1.2]{kedlaya-liu1}, this set is a subring of $W(L^\flat)$ on which 
the function $\left| \cdot \right|_r$ defined by
\[
\left| \sum_{n=0}^\infty p^n [\overline{x}_n]  \right|_r = \max_n \{p^{-n} \left| \overline{x}_n \right|^r\}
\]
is a complete multiplicative norm; this norm extends multiplicatively to $W^r(L^\flat)[p^{-1}]$.
For $0 < s \leq r$, let $\tilde{\calR}^{[s,r]}_L$ be the completion of $W^r(L^\flat)[p^{-1}]$
with respect to $\max\{\left| \cdot \right|_s, \left| \cdot \right|_r\}$,
and put $\tilde{\calR}^{[s,r]}_{L,A} = \tilde{\calR}^{[s,r]}_L \widehat{\otimes}_{\QQ_p} A$.
Let $\tilde{\calR}^r_{L,A}$ be the inverse limit of the $\tilde{\calR}^{[s,r]}_{L,A}$ over all $s \in (0,r)$, equipped with the Fr\'echet topology. Let $\tilde{\calR}_{L,A}$ be the direct  limit of the $\tilde{\calR}^r_{L,A}$ over all $r>0$, equipped with the locally convex direct limit topology (LF topology).

The notation is meant to suggest a strong analogy between (for example) the ring $\calR^{[s,r]}_{K,A}$ of power series convergent on a (relative) closed annulus and the somewhat more mysterious ring $\tilde{\calR}^{[s,r]}_{L,A}$. In fact, one may (somewhat imprecisely) think of the latter as consisting of certain ``Laurent series in $p$ with Teichm\"uller coefficients''; this point of view is pursued in \cite{kedlaya-witt} to express certain geometric consequences. In the case where $A$ is a field, a simultaneous development of
the ring-theoretic properties of $\calR^{[s,r]}_{K,A}$ and $\tilde{\calR}^{[s,r]}_{L,A}$
can be found in \cite{kedlaya-revisited}.
\end{defn}

For $A$ a field, the following is a consequence of \cite[Theorem~3.5.8]{kedlaya-liu2}.
\begin{lemma} \label{L:Witt module descend}
Let $L'$ be the completion of a (possibly infinite) Galois algebraic extension of $L$ with Galois group $G$. Then for $0 < s \leq r$, the functor from finite projective $\tilde{\calR}^{[s,r]}_{L,A}$-modules to finite projective
$\tilde{\calR}^{[s,r]}_{L',A}$-modules equipped with continuous semilinear $G$-actions is an equivalence of categories.
\end{lemma}
\begin{proof}
We first check full faithfulness.
Let $M,N$ be two finite projective $\tilde{\calR}^{[s,r]}_{L,A}$-modules
and put $P = M^\dual \otimes N$. Let $M',N',P'$ be the respective base extensions of $M,N,P$ to $\tilde{\calR}^{[s,r]}_{L',A}$, equipped with the induced $G$-actions. 
We then have maps
\[
P \cong \Hom(M,N) \to \Hom_G(M',N') \cong (P')^G,
\]
so to check full faithfulness we need only check that $P \to (P')^G$ is an isomorphism.
By writing $P$ as a direct summand of a finite free module, this reduces immediately to checking that 
\[
(\tilde{\calR}^{[s,r]}_{L',A})^G = \tilde{\calR}^{[s,r]}_{L,A}.
\]
For $A = \QQ_p$, this equality is a consequence of \cite[Theorem~9.2.15]{kedlaya-liu1};
we may deduce the general case from this by constructing a Schauder basis for $A$ over $\QQ_p$, as in \cite[Proposition~2.7.2/3]{bgr} or \cite[Lemma~2.2.9(b)]{kedlaya-liu1}.

We next check essential surjectivity.
Let $M'$ be a finite projective $\tilde{\calR}^{[s,r]}_{L',A}$-module equipped with a continuous semilinear $G$-action. The $G$-action may then be described in terms of an isomorphism $\iota$ between the two base extensions of $M'$ to $\tilde{\calR}^{[s,r]}_{L',A} \widehat{\otimes}_{\tilde{\calR}^{[s,r]}_{L,A}} \tilde{\calR}^{[s,r]}_{L',A}$;
note that $\iota$ obeys a cocycle condition expressing the compatibility of the $G$-action with composition in the group $G$. Specifying $\iota$ involves only finitely elements of the ring $\tilde{\calR}^{[s,r]}_{L',A}$, so it may be realized over some subfield of $L'$ which is the completion of an algebraic extension of $L$ of at most countable degree; we may thus assume that $L'$ itself has this form. From the proof of \cite[Theorem~9.2.15]{kedlaya-liu1}, we see that $\tilde{\calR}^{[s,r]}_{L',\QQ_p}$ splits in the category of Banach modules over $\tilde{\calR}^{[s,r]}_{L,\QQ_p}$; by tensoring with $A$, we see that
$\tilde{\calR}^{[s,r]}_{L',A}$ splits in the category of Banach modules over $\tilde{\calR}^{[s,r]}_{L,A}$.
This means that $\tilde{\calR}^{[s,r]}_{L,A} \to \tilde{\calR}^{[s,r]}_{L',A}$ is a universally injective morphism in the category of Banach modules over $\tilde{\calR}^{[s,r]}_{L,A}$, so we may apply a general descent theorem of Joyal--Tierney \cite{joyal-tierney}
(compare \cite[Lemma~1.2.17]{kedlaya-liu2}) to descend $M'$ to a finite projective $\tilde{\calR}^{[s,r]}_{L,A}$-module $M$. (See also \cite[Tag~08WE]{stacks-project} for a more elementary treatment of the corresponding descent statement for ordinary modules over a ring,
whose proof may be emulated for Banach modules.)
\end{proof}

\begin{remark} \label{R:glueing lemma}
It is shown in \cite[Theorem~3.2]{kedlaya-noetherian} that the rings $\tilde{\calR}^{[s,r]}_{L,\QQ_p}$ are \emph{really strongly noetherian}, that is, any affinoid algebra over such a ring (even in the sense of Berkovich) is noetherian. In particular, the rings $\tilde{\calR}^{[s,r]}_{L,A}$ are really strongly noetherian; consequently, they satisfy the analogues of Tate's acyclicity theorem
\cite[Theorem~7.14, Theorem~8.3]{kedlaya-adic}
and Kiehl's theorem on the characterization of coherent sheaves
\cite[Theorem~8.16]{kedlaya-adic}.
\end{remark}

\begin{defn}
Denote by $\varphi$ the following maps induced by the Witt vector Frobenius map on $W(L^\flat)$:
\[
W^r(L^\flat) \to W^{r/p}(L^\flat), \qquad
\tilde{\calR}^{[s,r]}_{L,A} \to \tilde{\calR}^{[s/p,r/p]}_{L,A},
\qquad
\tilde{\calR}^{r}_{L,A} \to \tilde{\calR}^{r/p}_{L,A}, \qquad
\tilde{\calR}_{L,A} \to \tilde{\calR}_{L,A}.
\]
A \emph{$\varphi$-module} over $\tilde{\calR}_{L,A}$ is a finite projective $\tilde{\calR}_{L,A}$-module equipped with a semilinear $\varphi$-action; unlike for $(\varphi, \Gamma)$-modules, this action is necessarily bijective (because the same is true of the maps $\varphi$ displayed above).
Let $\PhiMod_{L,A}$ be the category of $\varphi$-modules over $\tilde{\calR}_{L,A}$.
\end{defn}

\begin{lemma} \label{L:phi invariants}
For any $r,s$ with $0 < s \leq r/p$, we have
\[
\ker(\varphi-1: \tilde{\calR}^{[s,r]}_{L,A} \to \tilde{\calR}^{[s,r/p]}_{L,A}) = A.
\]
In particular, we have $\tilde{\calR}_{L,A}^{\varphi} = A$.
\end{lemma}
\begin{proof}
Again using \cite[Lemma~2.2.9(b)]{kedlaya-liu1}, we reduce to the case $A = \QQ_p$,
for which see \cite[Corollary~5.2.4]{kedlaya-liu1}. 
\end{proof}

\begin{lemma} \label{L:phi modules bundles}
For any $r,s$ with $0 < s \leq r/p$, the following categories are canonically equivalent:
\begin{enumerate}
\item[(a)]
the category $\PhiMod_{L,A}$;
\item[(b)]
the category of finite projective $\tilde{\calR}^r_{L,A}$-modules $M$ equipped with isomorphisms 
\[
\varphi^* M  \cong 
M \otimes_{\tilde{\calR}^r_{L,A}} \tilde{\calR}^{r/p}_{L,A};
\]
\item[(c)]
the category of finite projective $\tilde{\calR}^{[s,r]}_{L,A}$-modules
$M$ equipped with isomorphisms 
\[
\varphi^* M \otimes_{\tilde{\calR}^{[s/p,r/p]}_{L,A}} \tilde{\calR}^{[s,r/p]}_{L,A} \cong 
M \otimes_{\tilde{\calR}^{[s,r]}_{L,A}} \tilde{\calR}^{[s,r/p]}_{L,A}.
\]
\end{enumerate}
\end{lemma}
\begin{proof}
The functor from (b) to (a) is base extension from $\tilde{\calR}^r_{L,A}$ to $\tilde{\calR}_{L,A}$. The fact that it is an equivalence is an easy consequence of the bijectivity of the action of $\varphi$ on objects of $\PhiMod_{L,A}$.

The functor from (b) to (c) is base extension from $\tilde{\calR}^r_{L,A}$ to $\tilde{\calR}^{[s,r]}_{L,A}$. To prove that it is an equivalence, note first that (c) is formally equivalent to the same category with $r,s$ replaced by $r/p,s/p$. Using 
	Remark~\ref{R:glueing lemma}, we see additionally that (c) is equivalent to the same category no matter what values of $r,s$ are used. We may then check the equivalence between (b) and (c) by imitating the proof of \cite[Proposition~2.2.7]{kpx}.
\end{proof}

\begin{defn}
For $M \in \PhiMod_{L,A}$ and $n \in \ZZ$, define the twist $M(n)$ to have the same underlying module as $M$, but with the action of $\varphi$ multiplied by $p^{-n}$.
\end{defn}

\begin{lemma} \label{L:enough phi-invariants}
For $M \in \PhiMod_{L,A}$, there exists $n_0 \in \ZZ$ such that for all $n \geq n_0$,
$\varphi-1$ is surjective on $M(n)$ and its kernel generates $M(n)$ as a $\tilde{\calR}_{L,A}$-module.
\end{lemma}
\begin{proof}
The case $A = \QQ_p$ is treated in \cite[Proposition~6.2.2, Proposition~6.2.4]{kedlaya-liu1}; the same proofs carry over to the general case.
\end{proof}

\begin{theorem} \label{T:perfect embedding}
Let $\CC_L$ be a completed algebraic closure of $L$. Then the formula 
\[
V \mapsto (V \otimes_A \tilde{\calR}_{\CC_L, A})^{G_L}
\]
defines a full embedding $\Rep_A(G_L) \to \PhiMod_{L,A}$.
(We will discuss the essential image of this functor in \S\ref{sec:slopes}.)
\end{theorem}
\begin{proof}
The target of this functor is in $\PhiMod_{L,A}$ thanks to 
Lemma~\ref{L:Witt module descend} and Lemma~\ref{L:phi modules bundles}.
If $V \in \Rep_A(G_L)$ corresponds to $M \in \PhiMod_{L,A}$, 
then by Lemma~\ref{L:Witt module descend} there is a canonical $(\varphi, G_L)$-equivariant isomorphism
\[
V \otimes_A \tilde{\calR}_{\CC_L, A} \cong M \otimes_{\tilde{\calR}_{L,A}} \tilde{\calR}_{\CC_L,A}.
\]
By Lemma~\ref{L:phi invariants}, we may take $\varphi$-invariants to obtain an isomorphism
\[
V \cong (M \otimes_{\tilde{\calR}_{L,A}} \tilde{\calR}_{\CC_L,A})^{\varphi};
\]
from this we see that the functor $V \mapsto M$ is fully faithful.
\end{proof}

\begin{defn} \label{D:theta}
There is a canonical surjection $\theta: W(\frako_{L^\flat}) \to \frako_L$ whose kernel is a principal ideal; see
\cite[\S 1]{kedlaya-new-phigamma} for the construction.
By \cite[Lemma~5.5.5]{kedlaya-liu1}, for any interval $[s,r]$ containing 1,
this map extends to a surjection
$\theta: \tilde{\calR}^{[s,r]}_{L,A} \to L \widehat{\otimes}_{\QQ_p} A$.
\end{defn}

We now introduce a geometric construction developed in great detail by Fargues and Fontaine
\cite{fargues-fontaine}; see \cite{fargues, fargues-fontaine-durham} for expository treatments.
\begin{defn}
Define the graded ring
\[
P_{L,A} = \bigoplus_{n=0}^\infty P_{L,A,n}, \qquad P_{L,A,n} = \tilde{\calR}_{L,A}^{\varphi=p^n},
\]
and put $X_{L,A} = \Proj(P_{L,A})$.
Let $\VB_{L,A}$ be the category of quasicoherent locally finite free sheaves (or for short \emph{vector bundles}) on $X_{L,A}$.
\end{defn}

\begin{example}
For any $x$ in the maximal ideal of $\frako_{L^{\flat}}$, the sum $\sum_{n \in \ZZ} p^{-n} [\overline{x}^{p^n}]$
converges to a nonzero element of $P_{L,\QQ_p,1}$.
\end{example}

\begin{defn} \label{D:B-pairs}
The map $\theta$ defines a closed immersion $\Spec(L \widehat{\otimes}_{\QQ_p} A) \to X_{L,A}$; let $Z_{L,A}$ be the resulting closed subscheme of $X_{L,A}$, and let $U_{L,A}$ be the complement of $Z_{L,A}$ in $X_{L,A}$.
By \cite[Lemma~8.9.3]{kedlaya-liu1}, $U_{L,A}$ is affine and $Z_{L,A}$ is contained in an open affine subspace of $X_{L,A}$; consequently, we may complete $X_{L,A}$ along $Z_{L,A}$ to get another affine scheme $\widehat{Z}_{L,A}$. Let $\bB_{e,L,A}, \bB^+_{\dR,L,A}, \bB_{\dR,L,A}$ be the respective coordinate rings of the affine schemes 
\[
U_{L,A}, \widehat{Z}_{L,A}, U_{L,A} \times_{\Proj(P_{L,A})} \widehat{Z}_{L,A}.
\]
Let $\BPair_{L,A}$ be the category of glueing data for finite projective modules with respect to the diagram
\[
\bB_{e,L,A} \rightarrow \bB_{\dR,L,A} \leftarrow \bB_{\dR,L,A}^+.
\]
\end{defn}

\begin{theorem} \label{T:perfect equivalence}
The categories $\PhiMod_{L,A}$, $\BPair_{L,A}$, and $\VB_{L,A}$ are canonically equivalent.
\end{theorem}
\begin{proof}
The categories $\BPair_{L,A}$ and $\VB_{L,A}$ are equivalent by the Beauville-Laszlo theorem
\cite{beauville-laszlo} applied to the coordinate ring of some open affine subscheme of $X_{L,A}$ containing $Z_{L,A}$.
We construct the functor from $\VB_{L,A}$ to $\PhiMod_{L,A}$ as 
as in \cite[Definition~6.3.10]{kedlaya-liu1}.
Choose $\calF \in \VB_{L,A}$.
For each $f \in P_{L,A}$ which is homogeneous of positive degree,
we have an open affine subscheme of $X_{L,A}$ with coordinate ring $\tilde{\calR}_{L,A}[f^{-1}]^\varphi$; we may thus take sections of $\calF$ to obtain a finite projective module over  this ring. By base extension, we obtain a finite projective module over $\tilde{\calR}_{L,A}[f^{-1}]$ equipped with a semilinear $\varphi$-action. By \cite[Lemma~6.3.7]{kedlaya-liu1}, the possible values of $f$ generate the unit ideal in $\tilde{\calR}_{L,A}$, so we may glue on $\Spec(\tilde{\calR}_{L,A})$ to obtain an object of $\PhiMod_{L,A}$.

In the other direction, for $M \in \PhiMod_{L,A}$, we may view
$\bigoplus_{n=0}^\infty M(n)^{\varphi}$ as a graded module over $P_{L,A}$, and then form the associated quasicoherent sheaf, which we must show is a vector bundle.
Using Lemma~\ref{L:enough phi-invariants}, this follows as in the proof of \cite[Theorem~6.3.12]{kedlaya-liu1}.
\end{proof}

\begin{remark}
Throughout this remark, assume that $A$ is a field.
In this setting, the concept of $B$-pairs was introduced by Berger \cite{berger-b-pairs}
in a purely algebraic fashion,
without reference to the schemes defined in Definition~\ref{D:B-pairs}. Therein,
the ring $\bB_{e,L,A}$ appears in connection with Fontaine's crystalline period ring
$\bB_{\mathrm{crys}}$.

The scheme $X_{L,A}$ introduced by Fargues--Fontaine is in some sense a ``complete curve'': in particular, it is a regular one-dimensional noetherian scheme.
The space $P_{L,A,n}$ constitutes the sections of the $n$-th power of a certain ample line bundle on this scheme. This scheme admits something resembling an analytification in the category of adic spaces, in that there is a morphism into it from an adic space built out of the rings $\tilde{\calR}^{[s,r]}_{L,A}$, the pullback along which induces an equivalence of categories of coherent sheaves by analogy with Serre's GAGA theorem in complex algebraic geometry. This adic space in turn admits an infinite cyclic \'etale cover which is
a ``quasi-Stein space'' whose global sections are the ring $\bigcap_{r>0} \tilde{\calR}^r_{L,A}$, on which the deck transformations act via the powers of $\varphi$.
For more on this story, see the aforementioned references such as \cite{fargues-fontaine},
and also \cite[\S 8.7, 8.8]{kedlaya-liu1} and \cite[\S 4.7]{kedlaya-liu2}.
\end{remark}

\begin{remark} \label{R:deeply ramified}
Let $F/K$ be an algebraic extension. In the case where $F = \QQ_p(\mu_{p^\infty})$, we have seen already (Example~\ref{exa:cyclotomic}) that the completion of $F$ is a perfectoid field. 
This property turns to be closely related to ramification of local fields; let us now recall the precise nature of this relationship.

Coates--Greenberg \cite{coates-greenberg} define $F/K$ to be
\emph{deeply ramified} if for every finite extension $F'$ of $F$, the trace map
$\Trace: \frakm_{F'} \to \frakm_F$ is surjective. This holds in particular if $F/K$ is 
\emph{arithmetically profinite} in the sense of Fontaine--Wintenberger
\cite{fontaine-wintenberger} (see \cite[Corollary~1.5]{fesenko} for a detailed proof);
the latter holds in turn if $F/K$ is an infinite Galois extension with finite residual extension whose Galois group is a $p$-adic Lie group, by a theorem of Sen \cite{sen-lie}.
\end{remark}

\begin{lemma} \label{L:deeply ramified is perfectoid}
Let $F/K$ be an algebraic extension with completion $L$. Then $F/K$ is deeply ramified if and only if $L$ is a perfectoid field.
\end{lemma}
\begin{proof}
Suppose first that $L$ is not perfectoid;
this means that there exists $x \in \frako_F$ whose image in $\frako_F/(p)$ is not in the image of Frobenius,
and we will show that $\Trace: \frakm_{F(x^{1/p})} \to \frakm_F$ is not surjective.
For this purpose, there is no harm in replacing $F$ with a tamely ramified extension; we may thus assume
at once that $F$ admits no nontrivial tamely ramified extension (i.e., it is ``tamely closed'').

Let $c \geq p^{-1}$ denote the infimum of $\left| x - y^p \right|$ over all $y \in \frako_F$.
Choose $\epsilon > 1$
and choose $y_0 \in \frako_F$ such that $\left|x-y_0^p \right| \leq \min\{1, \epsilon^p c\}$;
note that $\left| x^{1/p} - y_0 \right| = \left| x - y_0^p \right|^{1/p}$.
Since $F$ is tamely closed, we may choose
$\mu \in \frako_F$ with 
\[
\left| x - y_0^p \right| \leq \left| \mu^p \right| \leq \min\{1, \epsilon^p \left| x - y_0^p \right|\}.
\]
Put $u := (x^{1/p} - y_0)/\mu$. Note that 
\[
\left| z-u \right| \geq \epsilon^{-2} \max\{1, \left| z \right|\} \qquad ( z \in F):
\]
for $z \notin \frako_F$ this is apparent because $\left| z \right| > 1 \geq \left |u \right|$,
while for $z \in \frako_F$ we have
\[
\left| z - u \right| = \left|\mu \right|^{-1} \left| y_0 + z \mu - x^{1/p} \right| 
\geq |\mu|^{-1} c^{1/p} \geq \epsilon^{-2}.
\]
Since $F$ is tamely closed, it follows that
for any $P(T) \in F[T]$ of degree at most $p-1$, $\left| P(u) \right|$ is at least $\epsilon^{-2(p-1)}$ times the Gauss norm of $P$.

Consider a general element $z = \sum_{i=0}^{p-1} z_i u^i \in \frakm_{F(x^{1/p})}$ with $z_0,\dots,z_{p-1} \in F$.
By the previous paragraph, $\max_i \{\left| z_i \right|\} \leq \epsilon^{2(p-1)}$.
Since $\Trace_{F(x^{1/p})/F}(x^{i/p}) = 0$ for $i=1,\dots,p-1$,
\[
\Trace_{F(x^{1/p})/F}(z) = \sum_{i=0}^{p-1} p z_i (-y_0/\mu)^i
\]
has norm at most $p^{-1} \left| \mu \right|^{1-p} \epsilon^{2(p-1)} 
\leq p^{-1} c^{(1-p)/p} \epsilon^{2(p-1)} 
\leq p^{-1/p}  \epsilon^{2(p-1)}$.
By taking $\epsilon$ sufficiently close to 1, we deduce that $\Trace: \frakm_{F(x^{1/p})} \to \frakm_F$ is not surjective.

Conversely, suppose that $L$ is perfectoid.
For any finite extension $F'$ of $F$, by Theorem~\ref{T:perfectoid}
the completion $L'$ of $F'$ is again perfectoid. 
Using the existence of a commutative diagram
\[
\xymatrix{
W(\frako_{L^\flat}) \ar[r] \ar[d] & \frako_{L} \ar[d] \\
W(\frako_{L^{\prime \flat}}) \ar[r] & \frako_{L'} 
}
\]
in which the horizontal arrows are surjective (see Definition~\ref{D:theta}), the surjectivity
of $\Trace: \frakm_{F'} \to \frakm_{F}$ reduces to the surjectivity of
$\Trace: \frakm_{L^{\prime \flat}} \to \frakm_{L^{\flat}}$,
or equivalently the fact that the cokernel of the latter map
is annihilated by all of $\frakm_{L^{\flat}}$.
This holds because the annihilator of the cokernel is nonzero
(because $L^{\prime \flat}/L^{\flat}$ is a finite separable extension)
and closed under taking $p$-th roots.
(This argument is a special case of the \emph{almost purity theorem} for perfectoid rings;
see \cite[Theorem~5.5.9]{kedlaya-liu1} and \cite[Theorem~7.9]{scholze1}.)
\end{proof}

\section{Slopes of \texorpdfstring{$\varphi$}{phi}-modules}
\label{sec:slopes}

We now introduce the important concept of \emph{slopes} of $\varphi$-modules. The basic theory is motivated by the corresponding theory of slopes of vector bundles on algebraic varieties (especially curves). In the process, we identify the essential image of the embedding functor of Theorem~\ref{T:perfect embedding} in case $A$ is a field.

\begin{hypothesis}
Throughout \S\ref{sec:slopes}, let $L$ be a perfectoid field.
\end{hypothesis}

\begin{lemma} \label{L:units}
Suppose that $A$ is a field.
Then $\tilde{\calR}_{L,A}^{\times} = \bigcup_{r>0} (W^r(L^\flat)[p^{-1}] \widehat{\otimes}_{\QQ_p} A)^\times$.
(This statement can be extended to the case where $A$ is reduced, but not more generally.)
\end{lemma}
\begin{proof}
See \cite[Corollary~4.2.5]{kedlaya-liu1}.
\end{proof}

\begin{defn}
Suppose that $A$ is a field. Let $v(A)$ denote the valuation group of $A$, normalized so that $v(\QQ_p^\times) = \ZZ$. 
Let $k$ be the largest finite extension of $\mathbb{F}_p$ which embeds into both $L$ and the residue field of $A$, and let $A_0$ be the unramified extension of $\QQ_p$ with residue field $k$; then $W(L^\flat) \widehat{\otimes}_{\QQ_p} A = W(L^\flat) \otimes_{\QQ_p} A$ splits into copies of the integral domain $W(L^\flat) \otimes_{A_0} A$ indexed by choices of the embedding $k \to L$.
The $p$-adic valuation on $W(L^\flat)$ extends to a valuation on
$W(L^\flat) \otimes_{A_0} A$ with values in $v(A)$; summing across components gives a map
$W(L^\flat) \otimes_{\QQ_p} A \to v(A) \cup \{+\infty\}$.
By Lemma~\ref{L:units},
we obtain a homomorphism $\tilde{\calR}_{L,A}^{\times} \to v(A)$; note that this map is invariant under $\varphi$-pullback.

By Theorem~\ref{T:perfect equivalence}, line bundles on $X_{L,A}$ correspond to $\varphi$-modules over $\tilde{\calR}_{L,A}$ whose underlying modules are projective of rank 1.
By taking determinants of these modules and using the $\varphi$-invariance of the map 
$\tilde{\calR}_{L,A}^{\times} \to v(A)$, we obtain a morphism $\deg: \Pic(X_{L,A}) \to v(A)$ called the \emph{degree map}.
(This map can also be given an interpretation in terms of rational sections of line bundles, in parallel with the usual construction of the degree map for line bundles on an algebraic curve; see \cite{fargues-fontaine} for this viewpoint.)
As usual, for $\calF \in \VB_{L,A}$ of arbitrary rank, we define the \emph{degree} of $\calF$ as the degree of its determinant $\wedge^{\rank(\calF)} \calF$, and (if $\calF \neq 0$) the \emph{slope} of $\calF$ as the ratio $\mu(\calF) = \deg(\calF)/\rank(\calF)$.
We may transfer these definitions to $\PhiMod_{L,A}$ using Theorem~\ref{T:perfect equivalence}.
\end{defn}

\begin{lemma} \label{L:no subinvariants}
For each positive integer $n$, we have
$\ker(p^n \varphi-1: \tilde{\calR}^{[s,r]}_{L,A} \to \tilde{\calR}^{[s,r/p]}_{L,A}) = 0$. \end{lemma}
\begin{proof}
Again using \cite[Lemma~2.2.9(b)]{kedlaya-liu1}, we reduce to the case $A = \QQ_p$.
Suppose that $x$ belongs to the kernel. The equality
\[
\left| x \right|_{t} = \left| \varphi(x) \right|_{t/p} = \left| p^{-n} x \right|_{t/p} 
= p^n \left| x \right|_{t/p}
\]
holds initially for $t \in [s, r]$, then by induction for 
$t \in [p^{-m} s, r]$ for each nonnegative integer $m$,
and hence for
all $t \in (0, r]$.
It follows that $\left| x \right|_t$ remains bounded as $t \to 0^+$, so by 
\cite[Lemma~4.2.4]{kedlaya-liu1} we have $x \in \tilde{\calR}_{L,A}^{\times}$.
However, by Lemma~\ref{L:units}, this means that $x$ has a well-defined $p$-adic valuation, which is the same as the valuation of $\varphi(x)$; we must then have $n=0$, contradiction.
\end{proof}

\begin{defn}
Suppose that $A$ is a field. For $\calF \in \VB_{L,A}$ nonzero, we say $\calF$ is \emph{stable} (resp. \emph{semistable}) if there does not exist a nonzero proper subobject $\calF'$ of $\calF$ such that $\mu(\calF') \geq \mu(\calF)$ (resp.\ $\mu(\calF') > \mu(\calF)$). 
For example, by Lemma~\ref{L:no subinvariants}, any rank 1 bundle is semistable.
We say $\calF$ is \emph{\'etale} if it is semistable of degree 0.

For general $\calF \in \VB_{L,A}$, Lemma~\ref{L:no subinvariants} implies that the set of slopes of nonzero subbundles of $\calF$ is bounded above; consequently, there exists a canonical filtration
\[
0 = \calF_0 \subset \calF_1 \subset \cdots \subset \calF_m = \calF
\]
such that the successive quotients $\calF_i/\calF_{i-1}$ are semistable and
$\mu(\calF_1/\calF_0) > \cdots > \mu(\calF_m/\calF_{m-1})$. 
This filtration is called the \emph{Harder--Narasimhan filtration}, or \emph{HN filtration}, of $\calF$. Note that $\calF_1$ is the maximal subbundle of $\calF$ achieving the maximal slope among nonzero subbundles of $\calF$. 
The \emph{HN polygon} of $\calF$ is the Newton polygon of length $\rank(\calF)$ in which the slope $\mu(\calF_i/\calF_{i-1})$ occurs with multiplicity $\rank(\calF_i)$; the total height of this polygon is $\deg(\calF)$.
\end{defn}

\begin{theorem} \label{T:etale properties}
Suppose that $A$ is a field.
\begin{enumerate}
\item[(a)] The tensor product of any two semistable bundles in $\VB_{L,A}$ is again semistable. In particular, the tensor product of two \'etale bundles is again \'etale.
\item[(b)]
Let $L'$ be any perfectoid field containing $L$. Then the HN polygon remains invariant under base extension from $\VB_{L,A}$ to $\VB_{L',A}$.
\item[(c)] The essential image of the full embedding $\Rep_A(G_L) \to \PhiMod_{L,A}$
in Theorem~\ref{T:perfect embedding} consists precisely of the \'etale objects.
\end{enumerate}
\end{theorem}
\begin{proof}
See \cite[\S 4]{kedlaya-liu1} and references therein.
\end{proof}

\begin{remark}
For general $A$, we may define the degree, rank, slope, and HN polygon of $\calF \in \VB_{L,A}$ as functions on $\Maxspec(A)$. Unfortunately, these functions do not extend well to the Berkovich space associated to $A$, because the theory of slopes behaves poorly when the degree map does not take discrete values.

In addition, if $A$ is not a field, then the subcategory of $\PhiMod_{L,A}$ consisting of pointwise \'etale objects may be strictly larger than the essential image of $\Rep_A(G_L)$; in fact, this already occurs for objects of rank 1, as noted in \cite[Remarque 4.2.10]{berger-colmez}.
For further discussion, see \cite{kedlaya-liu-families}.
\end{remark}

\begin{lemma} \label{L:bounded slopes}
For $\calF \in \VB_{L,A}$, the HN polygon of $\calF$, as a function on $\Maxspec(A)$, is bounded above and below, and its height is constant on connected components of $\Maxspec A$.
\end{lemma}
\begin{proof}
Using Lemma~\ref{L:enough phi-invariants},
we see that $\calF$ admits a surjective morphism from $\calO(n)^{\oplus d}$ for some integers $n,d$. It follows that the HN polygon of $\calF$ has no slopes less than $n$.
The same argument applies to the dual bundle shows that the HN polygon of $\calF$ also has slopes which are uniformly bounded above.
\end{proof}

\section{From \texorpdfstring{$\varphi$}{phi}-modules to \texorpdfstring{$(\varphi, \Gamma)$}{(phi, Gamma)}-modules}
\label{sec:alternate}

We now use $\varphi$-modules to give an alternate description of the category $\PhiGamma_{K,A}$ in the language of perfectoid fields. In the process, we will establish Theorem~\ref{T:phi gamma embedding}.

\begin{defn}
Let $L_K$ be the completion of $K(\mu_{p^\infty})$; it is a perfectoid field by
Example~\ref{exa:cyclotomic} and Theorem~\ref{T:perfectoid}.
Let $\PhiGammatilde_{K,A}$ denote the category of objects of $\PhiMod_{L_K,A}$ equipped with continuous semilinear $\Gamma_K$-actions, where $\Gamma_K = \Gal(K(\mu_{p^\infty})/K)$.
\end{defn}

We now complete the discussion initiated in Remark~\ref{R:not connected}.
\begin{remark} \label{R:not connected2}
Via the cyclotomic character, we may identify $\Gamma_F$ with $\Gamma$ and $\Gamma_K$ with an open subgroup of $\Gamma$. Put $\tilde{\calR}_{K,A} = \Ind_{\Gamma_K}^\Gamma \tilde{\calR}_{L_K,A}$; this is a direct sum of copies of $\tilde{\calR}_{L_K,A}$ indexed by the connected components of $K \otimes_F F(\mu_{p^\infty})$. 
We may then identify objects of $\PhiGammatilde_{K,A}$ with finite projective
$\tilde{\calR}_{K,A}$-modules equipped with continuous semilinear $\Gamma$-actions.
\end{remark}

\begin{defn} \label{D:untilde to tilde map}
Choose a coherent sequence $\zeta_p, \zeta_{p^2}, \dots$ of $p$-power roots of unity
and let $\epsilon$ be the element $(1, \zeta_p, \zeta_{p^2}, \dots) \in L_K^{\flat}$.
Then the map $W(k) \llbracket \pi \rrbracket \to W(\frako_{L^\flat_K})$ taking $\pi$ to $[\epsilon]-1$
is $(\varphi, \Gamma)$-equivariant; it thus extends to a $(\varphi, \Gamma)$-equivariant map $\calR_{K,A} \to \tilde{\calR}_{K,A}$.
\end{defn}

\begin{theorem} \label{T:tilde no tilde}
The categories $\PhiGamma_{K,A}$ and $\PhiGammatilde_{K,A}$ are equivalent via base extension along $\calR_{K,A} \to \tilde{\calR}_{K,A}$.
\end{theorem}
\begin{proof}
In the case $A = \QQ_p$, this is proved in \cite[Theorem~6.1.7]{kedlaya-liu2}.
It can also be deduced from prior results; for example, for $L$ a completed algebraic closure of $K$, Berger \cite[Th\'eor\`eme~A]{berger-b-pairs} constructed an equivalence between $\Phi\Gamma_{K,A}$ and the category of objects of $\BPair_{L,A}$ equipped with continuous semilinear $G_K$-actions. By Theorem~\ref{T:perfect equivalence}, these can be interpreted as objects of $\PhiMod_{L,A}$ equipped with continuous semilinear $G_K$-actions;
using Lemma~\ref{L:Witt module descend} and Lemma~\ref{L:phi modules bundles}, these can in turn be identified with objects of $\PhiGammatilde_{K,A}$.

To obtain full faithfulness in the general case, note that since both categories admit internal Homs in a compatible way, we reduce to checking that for $M \in \PhiGamma_{K,A}$, every $(\varphi, \Gamma)$-stable element $\bv \in M \otimes_{\calR_{K,A}} \tilde{\calR}_{K,A}$ belongs to $M$ itself. Using a Schauder basis for $A$ over $\QQ_p$ (see the proof of Lemma~\ref{L:Witt module descend}), we may construct a family of bounded $\QQ_p$-linear morphisms
$A \to \QQ_p$ whose kernels have zero intersection; by tensoring along these, we reduce the claim that $\bv \in M$ to a family of corresponding assertions in the previously treated case $A = \QQ_p$.

To obtain essential surjectivity, one may emulate the proof of \cite[Theorem~6.1.7]{kedlaya-liu2}; we only give a brief sketch here, as details will be given in upcoming work of Chojecki and Gaisin. By full faithfulness, we may reduce to the case $K = F$.
By Lemma~\ref{L:phi modules bundles}, we may start with a finite projective $\tilde{\calR}^{[s,r]}_{L_F,A}$-module $\tilde{M}$ equipped with an isomorphism of the base extensions of $\varphi^* \tilde{M}$ and $\tilde{M}$ to $\tilde{\calR}^{[s,r/p]}_{L_F,A}$, plus a compatible semilinear $\Gamma$-action.
Within $\tilde{\calR}^{[s,r]}_{L_F,A}$, we have a dense subring $\breve{\calR}^{[s,r]}_{F,A}$ consisting of the union of the closures of
the subrings $\varphi^{-n}((F \widehat{\otimes}_{\QQ_p} A)[\pi^{\pm}])$ for $n \geq 0$; it will suffice to descend $\tilde{M}$ to a finite projective $\breve{\calR}^{[s,r]}_{F,A}$-module $\breve{M}$
on which $\varphi$ and $\Gamma$ continue to act. Using the density of
$\breve{\calR}^{[s,r]}_{F,A}$ in $\tilde{\calR}^{[s,r]}_{L_F,A}$, we may apply
\cite[Lemma 5.6.8]{kedlaya-liu2} to descend the underlying module of $\tilde{M}$, but the resulting descended module will typically not be $\Gamma$-stable; this may corrected using a sequence of successive approximations as in \cite[Lemma 5.6.9]{kedlaya-liu2}. We thus obtain a $\Gamma$-stable descended module, which is then easily shown to be also $\varphi$-stable.
\end{proof}

As a corollary, we may now establish Theorem~\ref{T:phi gamma embedding1}.
\begin{proof}[Proof of Theorem~\ref{T:phi gamma embedding1}]
By Theorem~\ref{T:tilde no tilde}, it suffices to exhibit a full embedding
$\Rep_A(G_K) \to \PhiGammatilde_{K,A}$. We obtain this embedding from
Theorem~\ref{T:perfect embedding} by adding $\Gamma_K$-descent data.
\end{proof}

\begin{defn}
We say that $M \in \PhiGamma_{K,A}$ is \emph{\'etale} if its image in $\PhiMod_{L_K,A}$ is \'etale. 
\end{defn}

\begin{theorem}
Suppose that $A$ is a field. Then $M \in \PhiGamma_{K,A}$ is \'etale if and only if it belongs
to the essential image of the functor $\Rep_A(G_K)\to \PhiGamma_{K,A}$.
\end{theorem}
\begin{proof}
This is immediate from Theorem~\ref{T:etale properties}(c).
\end{proof}

\begin{remark}
Theorem~\ref{T:tilde no tilde}, when restricted to \'etale objects, reproduces the Cherbonnier-Colmez theorem on the overconvergence of $p$-adic representations \cite{cherbonnier-colmez}. However, the proof we have in mind is closer in spirit to the one in \cite[\S 2]{kedlaya-new-phigamma}.
\end{remark}

\begin{remark}
Note that the embedding in Theorem~\ref{T:phi gamma embedding1} is not quite canonical: it depends on the coherent sequence of $p$-power roots of unity chosen in Definition~\ref{D:untilde to tilde map}. This suggests that in some sense, the embedding of $\Rep_{G_K}(A)$ into $\PhiGammatilde_{K,A}$  is more natural than the embedding into $\PhiGamma_{K,A}$.
\end{remark}

\begin{remark}
Let $F/K$ be a deeply ramified Galois algebraic extension, so that by Lemma~\ref{L:deeply ramified is perfectoid} the completion $L$ of $F$ is a perfectoid field. 
Using Lemma~\ref{L:Witt module descend}, we may describe the category
$\PhiGammatilde_{K,A}$ as the category of objects of $\PhiMod_{L,A}$ equipped with continuous semilinear $\Gal(F/K)$-actions; in particular, we again obtain a full embedding
of $\Rep_A(G_K)$ into this category. However, in general there is no natural analogue of the category $\PhiGamma_{K,A}$ because the ring $\tilde{\calR}_{L,A}$ cannot be obtained in a natural way from a ring of Laurent series. For this reason, we are driven to reformulate known constructions involving $\PhiGamma_{K,A}$ using $\PhiGammatilde_{K,A}$ with the eye towards generalizing to towers other than the cyclotomic tower; we will pick up on this theme
in \S\ref{sec:coda}.
\end{remark}

\section{Cohomology of $(\varphi, \Gamma)$-modules}

We now upgrade the previous discussion to relate Galois cohomology to $(\varphi, \Gamma)$-modules. This time, we start directly with the perfectoid framework.
\begin{defn}
Let $L$ be a perfectoid field.
For $\tilde{M} \in \PhiMod_{L,A}$, let $H^0_{\varphi}(\tilde{M}), H^1_{\varphi}(\tilde{M})$ be the kernel and cokernel of $\varphi-1$ on $\tilde{M}$, and put $H^i_{\varphi}(\tilde{M}) = 0$ for $i>1$.
\end{defn}

\begin{lemma} \label{L:truncate cohomology}
Choose $r,s$ with $0 < s \leq r/p$. Let $\tilde{M}, \tilde{M}^r, \tilde{M}^{[s,r]}$
be corresponding objects in the categories (a),(b),(c) of Lemma~\ref{L:phi modules bundles}.
Then the vertical arrows in the diagram
\[
\xymatrix@C=2cm{
0 \ar[r] & \tilde{M} \ar^(.4){\varphi-1}[r] & \tilde{M} \ar[r] &  0 \\
0 \ar[r] & \tilde{M}^r \ar^(.4){\varphi-1}[r] \ar[u] \ar[d] & \tilde{M}^r \otimes_{\tilde{\calR}^r_{L,A}} \tilde{\calR}^{r/p}_{L,A} \ar[r] \ar[d] \ar[u] & 0 \\
0 \ar[r] & \tilde{M}^{[s,r]} \ar^(.4){\varphi-1}[r] & \tilde{M}^{[s,r]} \otimes_{\tilde{\calR}^{[s,r]}_{L,A}} \tilde{\calR}^{[s,r/p]}_{L,A} \ar[r] & 0
}
\]
constitute quasi-isomorphisms of the horizontal complexes.
\end{lemma}
\begin{proof}
The proof of \cite[Proposition~6.3.19]{kedlaya-liu1} in the case $A = \QQ_p$ adapts without change.
\end{proof}

\begin{defn}
For $G$ a profinite group acting continuously on a topological abelian group $M$,
let $C(G,M)$ denote the complex of inhomogeneous continuous cochains on $G$ with values in $M$. Denote by $H^i(G,M)$ or $H^i_G(M)$ the cohomology groups of this complex.

For $M \in \PhiGamma_{K,A}$, let $C_{\varphi, \Gamma}(M)$ denote the total complex associated to the double complex
\[
0 \to C(\Gamma, M) \stackrel{\varphi-1}{\to} C(\Gamma,M) \to 0.
\]
Denote by $H^i_{\varphi, \Gamma}(M)$ the cohomology groups of this complex.
We make an analogous definition for $\tilde{M} \in \PhiGammatilde_{K,A}$.
\end{defn}

\begin{lemma} \label{L:perfect descend cohomology}
Let $L$ be a perfectoid field.
Let $L'$ be the completion of a (possibly infinite) Galois algebraic extension of $L$ with Galois group $G$. Then for $0 < s \leq r$, for $M$ a finite projective $\tilde{\calR}^{[s,r]}_{L,A}$-module, $M$ is $G$-acyclic, i.e.,
the morphism $M \to C(G,M)$ is a quasi-isomorphism.
\end{lemma}
\begin{proof}
Using a Schauder basis for $A$ over $\QQ_p$ (see the proof of Lemma~\ref{L:Witt module descend}), we may reduce to the case $A = \QQ_p$,
for which we may apply \cite[Theorem~8.2.22]{kedlaya-liu1}.
\end{proof}

\begin{theorem}
Let $L$ be a perfectoid field.
For $\tilde{M} \in \PhiMod_{L,A}$, $\calF \in \VB_{L,A}$ corresponding as in
Theorem~\ref{T:perfect equivalence}, we have canonical identifications
$H^i_{\varphi}(\tilde{M}) \cong H^i(X_{L,A}, \calF)$ for all $i \geq 0$.
\end{theorem}
\begin{proof}
Since $X_{L,A}$ is separated and is covered by two open affine subschemes, we have 
$H^i(X_{L,A}, \calF) = 0$ for $i \geq 2$. The identification for $i=0,1$ arises directly
from Theorem~\ref{T:perfect equivalence}, using in the case $i=1$ the interpretation of the cohomology groups as Yoneda extension groups.
\end{proof}

\begin{theorem} \label{T:compare phi Gamma cohomology}
For $M \in \PhiGamma_{K,A}$, $\tilde{M} \in \PhiGammatilde_{K,A}$ corresponding via Theorem~\ref{T:tilde no tilde}, 
the morphisms $H^i_{\varphi, \Gamma}(M) \to H^i_{\varphi, \Gamma}(\tilde{M})$ are isomorphisms for all $i$.
\end{theorem}
\begin{proof}
As in Theorem~\ref{T:tilde no tilde},
the case $A = \QQ_p$ is treated in \cite[Theorem~6.1.7]{kedlaya-liu2}, and we sketch
an adaptation to the general case and refer to upcoming work of Chojecki and Gaisin for further details. We again reduce to the case $K = F$, and to calculating $(\varphi, \Gamma)$-cohomology for a pair of modules $\breve{M}, \tilde{M}$ in which 
$\breve{M}$ is finite projective over $\breve{\calR}^{[s,r]}_{F,A}$
and $\tilde{M}$ is the base extension to $\tilde{\calR}^{[s,r]}_{L_F,A}$.
In this setting, we may already show that $H^i_{\Gamma}(\breve{M}) = H^i_{\Gamma}(\tilde{M})$ using the method of \cite[Lemma 5.6.6]{kedlaya-liu2}, i.e., by first making a direct calculation in the case where $\breve{M} = \breve{\calR}^{[s,r]}_{F,A}$,
then using this case to make a series of successive approximations in the general case.
\end{proof}

\begin{theorem} \label{T:coherence cohomology}
Suppose that $[K:\QQ_p] < \infty$.
For $M \in \PhiGamma_{K,A}$, we have the following.
\begin{enumerate}
\item[(a)] The groups $H^i_{\varphi, \Gamma}(M)$ are finite $A$-modules for $i=0,1,2$, and vanish for $i>2$.
\item[(b)] For any morphism $A \to B$ of affinoid algebras over $\QQ_p$, the canonical morphism
\[
C_{\varphi, \Gamma}(M) \otimes^{\mathbb{L}}_A B
\to C_{\varphi, \Gamma}(M \otimes_{\calR_{K,A}} \calR_{K,B})
\]
is a quasi-isomorphism.
\item[(c)]
If $M$ is the image of $V \in \Rep_A(G_K)$ under Theorem~\ref{T:phi gamma embedding1},
then there is a canonical quasi-isomorphism $C(G_K, V) \cong C_{\varphi, \Gamma}(M)$.
In particular, the $A$-modules $H^i_{\varphi, \Gamma}(M)$ coincide with the Galois cohomology groups of $V$.
\end{enumerate}
\end{theorem}
\begin{proof}
See \cite[Proposition~2.3.7, Theorem~4.4.2, Theorem~4.4.3]{kpx}.
\end{proof}

\begin{remark} \label{R:Cartan-Serre}
While Theorem~\ref{T:compare phi Gamma cohomology} and Theorem~\ref{T:coherence cohomology} together assert that the groups $H^i_{\varphi, \Gamma}(\tilde{M})$
are finite $A$-modules,
the proof of this statement depends crucially on the interpretation of $\tilde{M}$ in terms of the category $\PhiGamma_{K,A}$. 
To illustrate this, we sketch a proof of Theorem~\ref{T:coherence cohomology}(a) in the spirit of \cite{kpx} but technically somewhat simpler. (It is also slightly more general,
as we only need to assume that $A$ is a noetherian Banach algebra over $\QQ_p$.)

We first reduce to the case $K = \QQ_p$ using the $(\varphi,\Gamma)$-module-theoretic counterpart of Shapiro's lemma described in \cite[Theorem~3.2]{liu-herr}
(compare also Remark~\ref{R:not connected2}).
For $M \in \PhiGamma_{A,K}$, Shapiro's lemma as usual implies that
$H^i_{\varphi,\Gamma_K}(M) \cong H^i_{\varphi,\Gamma}(\Ind^{\Gamma}_{\Gamma_K} M)$;
since the definition of the latter does not explicitly reference the module structure,
we may view $\Ind^\Gamma_{\Gamma_K} M$ as a module over
$\Ind^{\Gamma}_{\Gamma_K} \calR_{K,A}$ and then restrict scalars to $\calR_{\QQ_p,A}$
without changing the cohomology.

Now assuming $K = \QQ_p$, choose $r,s,r',s'$ with $0 < s < s' \leq r'/p \leq r/p$. 
Let $\calR^{[s,r]}_{\QQ_p,A}$ be the ring of rigid analytic functions 
on the disc $p^{-rp/(p-1)} < \left| \pi \right| < p^{-sp/(p-1)}$ over $A$. For $r$ sufficiently small, we may (by analogy with Lemma~\ref{L:phi modules bundles} and Lemma~\ref{L:truncate cohomology}) realize $M$ as a finite projective $\calR^{[s,r]}_{\QQ_p,A}$-module $M^{[s,r]}$
equipped with an isomorphism 
\[
\varphi^* M^{[s,r]} \otimes_{\calR^{[s/p,r/p]}_{\QQ_p,A}} \calR^{[s,r/p]}_{\QQ_p,A} 
\cong M^{[s,r]} \otimes_{\calR^{[s,r]}_{\QQ_p,A}} \calR^{[s,r/p]}_{\QQ_p,A}
\]
and compute $H^i_{\varphi, \Gamma}(M)$ as the cohomology of the total complex
\[
0 \to C(\Gamma, M^{[s,r]}) \stackrel{\varphi-1}{\to} C(\Gamma, M^{[s,r/p]}) \to 0.
\]
We then have a diagram
\[
\xymatrix{
0 \ar[r] & C(\Gamma, M^{[s,r]}) \ar[d] \ar^{\varphi-1}[r] & C(\Gamma, M^{[s,r/p]}) \ar[r] \ar[d] & 0 \\
0 \ar[r] & C(\Gamma, M^{[s',r']}) \ar^{\varphi-1}[r] & C(\Gamma, M^{[s',r'/p]}) \ar[r] & 0
}
\]
in which the vertical arrows define a quasi-isomorphism of the total complexes associated to the rows. However, each vertical arrow is composed of maps which are  \emph{completely continuous} morphisms of Banach spaces over $A$, i.e., uniform limits of morphisms of finite rank. 
By the Cartan-Serre-Schwartz lemma as applied in \cite[\S 3]{kedlaya-liu-finiteness}
(compare \cite[Satz~2.6]{kiehl-finiteness}),
we deduce that the cohomology groups of the total complexes are finite $A$-modules.
(By contrast, the maps $\tilde{\calR}^{[s,r]}_{\QQ_p,A} \to \tilde{\calR}^{[s',r']}_{\QQ_p,A}$ are not completely continuous.)
\end{remark}

\begin{remark}
In Remark~\ref{R:Cartan-Serre}, 
note that $\varphi$ is only $A$-linear rather than $(F \widehat{\otimes}_{\QQ_p} A)$-linear; we thus need $[K:\QQ_p] < \infty$ in order to reduce to the case $K = F = \QQ_p$.
If we relax the hypothesis on $K$ to allow it to be a more general local field, the vanishing of $H^i_{\varphi, \Gamma}(M)$ for $i>2$ and finiteness for $i=0$ remains valid, but the finite generation of $H^i_{\varphi,\Gamma}$ for $i=1$ and $i=2$ can fail.
\end{remark}

\section{The cyclotomic deformation}

We now consider a key example of an arithmetic deformation. The construction follows
\cite[Definition~4.4.7]{kpx}, but we opt here for  more geometric language.

\begin{defn}
Let $X$ be a rigid analytic space over $\QQ_p$. 
The rings $\calR_{K,A}, \tilde{\calR}^{[s,r]}_{K,A}, \tilde{\calR}^{r}_{K,A}, \tilde{\calR}_{K,A}$ all satisfy the sheaf axiom and Tate acyclicity with respect to finite coverings by affinoid subdomains: for example, the \v{C}ech sequence for $\tilde{\calR}_{K,A}$
with respect to a given covering is obtained from the corresponding sequence for the structure sheaf by the exact operation of 
taking the completed tensor product over $\QQ_p$ with $\tilde{\calR}_{K,\QQ_p}$.
(The exactness of completed tensor products over $\QQ_p$ does involve a nontrivial argument using Schauder bases; see for example \cite[Lemma~2.2.9]{kedlaya-liu1}.)
These constructions thus give rise to ring sheaves $\calR_{K,X}, \tilde{\calR}^{[s,r]}_{K,X}, \tilde{\calR}^{r}_{K,X}, \tilde{\calR}_{K,X}$ on the affinoid space $\Maxspec(A)$ which are acyclic on affinoid subspaces.

Let $\PhiGamma_{K,X}, \PhiGammatilde_{K,X}$ be the categories of finite projective
modules over the respective ring sheaves $\calR_{K,X}$, $\tilde{\calR}_{K,X}$ equipped with continuous commuting semilinear actions of $\varphi, \Gamma$. 
These form stacks for both the analytic topology and the \'etale topology; in particular, Theorem~\ref{T:tilde no tilde} gives rise to an equivalence of categories $\PhiGamma_{K,X} \to \PhiGammatilde_{K,X}$.
\end{defn}

\begin{defn}
Let $\ZZ_p \llbracket \Gamma_K \rrbracket$ be the completed group algebra.
Since this ring is formally of finite type over $\ZZ_p$, we may apply Berthelot's generic fiber construction (see for example \cite[\S 7]{dejong-crys})
to view this ring as the collection of bounded-by-1 rigid analytic functions on a certain one-dimensional quasi-Stein space $W_K$ over $\QQ_p$ (the \emph{weight space} of $\Gamma_K$).
More precisely, in case $\Gamma_K \cong \ZZ_p$, the space $W_K$ is an open unit disc admitting $\gamma-1$ as a coordinate for any topological generator $\gamma \in \Gamma_K$; in the general case, $W_K$ is a finite disjoint union of such discs.

The action of $\ZZ_p \llbracket \Gamma_K \rrbracket$ on $\Lambda_K = \calO(W_K)$ by (left) multiplication
defines a canonical one-dimensional Galois representation on $W_K$; let $\Dfm_{K}$ be the corresponding $(\varphi, \Gamma_K)$-module.
For $X$ a rigid analytic space over $\QQ_p$ and $M \in \PhiGamma_{K,X}$,
define the \emph{cyclotomic deformation} of $M$
as the external tensor product $M \boxtimes \Dfm_{K} \in \PhiGamma_{K,X \times_{\Qp} W_K}$. We similarly define the cyclotomic deformation of $\tilde{M} \in \PhiGammatilde_{K,X}$ as an object $\tilde{M} \boxtimes \Dfm_{K}$ of $\PhiGammatilde_{K,X \times_{\Qp} W_K}$.
\end{defn}

\begin{remark} \label{R:Gamma acyclic}
For $M \in \PhiGamma_{K,A}$, we may view the cyclotomic deformation of $M$ as arising from the completed tensor product $M' = M \widehat{\otimes}_{\Qp} \Lambda_K$. Since the action of $\Gamma_K$ on $M$ is locally analytic, the action map $M \times \Gamma_K \to M$ extends to a projection $M' \to M$ sectioning the inclusion $M \to M'$. From the existence of this projection, it follows (as in the proof of \cite[Theorem~4.4.8]{kpx}) that 
\[
(M')^{\Gamma_K} = 0, \qquad (M')_{\Gamma_K} \cong M.
\]
By contrast, for $\tilde{M} \in \PhiGammatilde_{K,A}$, the action map
$\tilde{M} \times \Gamma_K \to \tilde{M}$ is continuous, but does not define a bounded map
$\tilde{M} \widehat{\otimes}_{\QQ_p} \Lambda_K \to \tilde{M}$.
\end{remark}

\begin{defn}
Define the ring $\breve{\calR}_{K,A} = \bigcup_{n=0}^\infty \varphi^{-n}(\calR_{K,A})$.
We may then define the category $\PhiGammabreve_{K,A}$ of $(\varphi, \Gamma)$-modules over $\breve{\calR}_{K,A}$. The base extension functor $\PhiGamma_{K,A} \to \PhiGammabreve_{K,A}$ is obviously surjective; by Theorem~\ref{T:tilde no tilde} it is also fully faithful, and hence an equivalence.
\end{defn}

\section{Iwasawa cohomology and the cyclotomic deformation}

The goal of this section is to describe various constructions in the classical language of $(\varphi, \Gamma)$-modules which play a role in Iwasawa theory, then translate these into the other categories so as to isolate the role of the cyclotomic extension.
Assume hereafter that $[K:\QQ_p] < \infty$.

\begin{defn} \label{D:psi}
Define the map $\psi: \calR_{K,A} \to \calR_{K,A}$ as the reduced trace of $\varphi$, i.e.,
as $p^{-1} \varphi^{-1} \circ \Trace(\calR_{K,A} \to \varphi(\calR_{K,A}))$;
by definition, it is a left inverse of $\varphi$. For any $M \in \PhiGamma_{K,A}$,
we may likewise take the reduced trace of the action of $\varphi$ on $M$ to obtain an action of $\psi$ on $M$, which is again a left inverse of $\varphi$;
concretely, the action of $\psi$ on $M$ is characterized by additivity and the identity
\[
\psi(r \varphi(\bv)) = \psi(r) \bv.
\]
We have an exact sequence
\begin{equation} \label{eq:psi sequence}
0 \to M^{\varphi=1} \to M^{\psi=1} \stackrel{\varphi-1}{\longrightarrow} M^{\psi=0}.
\end{equation}
\end{defn}

\begin{prop} \label{P:psi finite}
For any $M \in \PhiGamma_{K,A}$,
the $A$-module $M/(\psi-1)$ is finitely generated.
\end{prop}
\begin{proof}
It suffices to treat the case $K = \QQ_p$, for which see \cite[Proposition~3.3.2(1)]{kpx}.
\end{proof}

\begin{cor}
For any $M \in \PhiGamma_{K,A}$, there is a canonical isomorphism
\[
M^{\varphi=1} \cong \Hom_A(M^*/(\psi-1), A).
\]
In particular, the $A$-module $M^{\varphi=1}$ is finitely generated.
\end{cor}
\begin{proof}
It suffices to treat the case $K = \QQ_p$.
In this case, using residues of power series, one constructs as in \cite[Notation~2.3.13]{kpx} a nondegenerate pairing $M \times M^* \to \calR_{K,A}(1)$ satisfying
\[
\{\varphi(x), \varphi(y)\} = \{x, y\} \qquad (x \in M, y \in M^*)
\]
and hence
\begin{equation}
 \label{eq:pairing phi psi}
\{\varphi(x), y\}= \{x, \psi(y)\} \qquad (x \in M, y \in M^*).
\end{equation}
We will show that the map $M \to \Hom_A(M^*, A)$ arising from the pairing $\{-,-\}$ induces the desired isomorphism; this will then imply the finite generation of $M^{\varphi=1}$
using Proposition~\ref{P:psi finite}.

To begin with, the nondegeneracy of the pairing $\{-, -\}$ implies the injectivity of $M \to \Hom_A(M^*, A)$, and
the identity \eqref{eq:pairing phi psi} shows that the image of this map is contained in
$\Hom_A(M^*/(\psi-1), A)$. In the other direction, note that Proposition~\ref{P:psi finite} and the open mapping theorem imply that $(\psi-1)M^*$ is a closed subspace of $M^*$
for the Fr\'echet topology, so every element of $\Hom_A(M^*/(\psi-1), A)$ defines a continuous $A$-linear map $M^* \to A$. By the perfectness of the pairing,
any such map corresponds to an element of $M$.
\end{proof}

\begin{remark} \label{R:berger explicit}
One of the key constructions in $p$-adic Hodge theory is Fontaine's \emph{crystalline period functor}, which takes a representation $V \in \Rep_A(G_K)$ to the finite $A$-module
\[
D_{\mathrm{crys}}(V) = (V \otimes_{\QQ_p} \bB_{\mathrm{crys}})^{G_K}
\]
where $\bB_{\mathrm{crys}}$ is a certain topological $\QQ_p$-algebra (the \emph{ring of crystalline periods}). For example, when $A = \QQ_p$ and $V$ is the $p$-adic \'etale cohomology of a smooth proper scheme over $\frako_K$, the \emph{crystalline comparison theorem} defines a functorial isomorphism of $D_{\mathrm{crys}}(V)$ with the rational crystalline cohomology of the same scheme; see \cite{bhatt-morrow-scholze} for a thoroughly modern take on the construction.

The functor $D_{\mathrm{crys}}$ factors naturally through $\PhiGamma_{K,A}$
and $\PhiGammatilde_{K,A}$: for example, for $M = \PhiGamma_{K,A}$,
$\tilde{M} \in \PhiGammatilde_{K,A}$
corresponding to $V \in \Rep_A(G_K)$, we have a canonical isomorphism
\[
D_{\mathrm{crys}}(V) \cong M[t^{-1}]^\Gamma \cong \tilde{M}[t^{-1}]^{\Gamma}, \qquad
t = \log(1+\pi).
\]
In the case of $\tilde{M}$, all we are using about $t$ is that it belongs to $P_{L_K, A, n}$
for some $n>0$ and that its zero locus on $X_{L,A}$ is precisely $Z_{L,A}$; this interpretation can be used to avoid specific references to the cyclotomic tower.

When $K/\QQ_p$ is unramified, $A$ is finite over $\QQ_p$, and $M \in \PhiGamma_{K,A}$ is crystalline (i.e., its $D_{\mathrm{crys}}$ is ``as large as possible'', as if $M$ arose from the comparison isomorphism), the object $M^{\psi=1}$ is related to the Galois
cohomology $H^1$ of each of the twists of $M$ in its cyclotomic
deformation (see Corollary~\ref{C:Iwasawa} below), and $M^{\psi=0}$ is related to the $D_{\mathrm{crys}}$ of the same twists
of $M$.  As shown by Berger \cite{berger-explicit} (and generalized by Nakamura \cite{nakamura}),
explicit formulas for Bloch--Kato's and Perrin-Riou's exponential maps,
and the ``$\delta(M)$'' formula for the determinant of the latter,
follow from a study of the relationship between these two objects.  We
will therefore focus on describing corresponding
objects made from $\tilde{M}$ in $\PhiGammatilde_{K,A}$ for general $K,A,M$.
\end{remark}

\begin{defn} \label{D:psi cohomology}
For $M \in \PhiGamma_{K,A}$, 
let $C_\psi(M)$ denote the complex
\[
0 \to M \stackrel{\psi-1}{\to} M \to 0
\]
with the nonzero terms placed in degrees $1$ and $2$.
Denote by $H^i_{\psi}(M)$ the cohomology groups of this complex.

Let $C_{\psi, \Gamma}(M)$ denote the total complex associated to the double complex
\[
0 \to C(\Gamma, M) \stackrel{\psi-1}{\to} C(\Gamma,M) \to 0.
\]
Denote by $H^i_{\psi, \Gamma}(M)$ the cohomology groups of this complex. The diagram
\[
\xymatrix{
0 \ar[r] & C(\Gamma, M) \ar^{\varphi-1}[r] \ar^{\id}[d] & C(\Gamma,M) \ar^{-\psi}[d]\ar[r] & 0 \\
0 \ar[r] & C(\Gamma, M) \ar^{\psi-1}[r] & C(\Gamma,M) \ar[r] & 0
}
\]
induces a morphism $C_{\varphi, \Gamma}(M) \to C_{\psi, \Gamma}(M)$
which is a quasi-isomorphism \cite[Proposition~2.3.6]{kpx}.
\end{defn}

\begin{defn}
For $M \in \PhiGamma_{K,X}, \tilde{M} \in \PhiGammatilde_{K,X}$ corresponding via
Theorem~\ref{T:tilde no tilde},
define the following sheaves on $X$:
\begin{align*}
\calH^i_{\psi, \Gamma}(M): \Maxspec(B) &\mapsto H^i_{\psi, \Gamma}(M \otimes \calR_{K,B}) \\
\calH^i_{\varphi, \Gamma}(M): \Maxspec(B) &\mapsto H^i_{\varphi, \Gamma}(M \otimes \calR_{K,B}) \\
\calH^i_{\varphi, \Gamma}(\tilde{M}): \Maxspec(B) &\mapsto H^i_{\varphi, \Gamma}(\tilde{M} \otimes \tilde{\calR}_{K,B}).
\end{align*}
By Theorem~\ref{T:compare phi Gamma cohomology} and
Definition~\ref{D:psi cohomology}, the sheaves $\calH^i_{\psi, \Gamma}(M),
\calH^i_{\varphi,\Gamma}(M), \calH^i_{\varphi,\Gamma}(\tilde{M})$
are canonically isomorphic; by 
\cite[Theorem~4.4.3, Remark~4.4.4]{kpx}, they are coherent.
\end{defn}

\begin{theorem} \label{T:psi to deformation}
For $M \in \Phi\Gamma_{K,A}$, with notation as in Remark~\ref{R:Gamma acyclic},
there is a canonical morphism of complexes
\[
C_{\psi, \Gamma}(M') \to C_{\psi}(M)
\]
which is a quasi-isomorphism.
\end{theorem}
\begin{proof}
Apply \cite[Theorem~4.4.8]{kpx}.
\end{proof}
\begin{cor} \label{C:psi to deformation}
Suppose that $M \in \PhiGamma_{K,A}$, $\tilde{M} \in \PhiGammatilde_{K,A}$ correspond as in 
Theorem~\ref{T:tilde no tilde}. Then for $X = \Maxspec(A)$,
we have canonical isomorphisms
\[
\Gamma(X \times_K W_K, \calH^i_{\psi, \Gamma}(\tilde{M} \boxtimes \Dfm_K)) \cong H^i_\psi(M).
\]
\end{cor}

This statement applies to Iwasawa cohomology as follows.
\begin{cor} \label{C:Iwasawa}
For $V \in \Rep_A(G_K)$ corresponding to $M \in \PhiGamma_{K,A}$,
$\tilde{M} \in \PhiGammatilde_{K,A}$ via 
Theorem~\ref{T:phi gamma embedding}
and Theorem~\ref{T:tilde no tilde},
write
\[
H^i_{\Iw}(G_K, V) = \left( \lim_{n \to \infty} H^i(G_{K(\mu_{p^n})}, T) \right) \otimes_{\ZZ} \QQ
\]
for $T \subseteq V$ the unit ball for some Galois-invariant Banach module norm on $V$
(the construction does not depend on the choice), with the transition maps being the corestriction maps. Then for each $i$, we have functorial isomorphisms
\[
H^i_{\Iw}(G_K, V) \widehat{\otimes}_{\ZZ_p\llbracket \Gamma_K \rrbracket} \Lambda_K \cong H^i_\psi(M) \cong H^i_{\varphi, \Gamma}(\tilde{M} \boxtimes \Dfm_K)
\]
of $\Lambda_K$-modules compatible with base change.
\end{cor}
\begin{proof}
Combine Theorem~\ref{T:psi to deformation} with \cite[Corollary~4.4.11]{kpx}.
\end{proof}

\begin{remark}
Corollary~\ref{C:Iwasawa} is a variant of a statement made by Fontaine in his original language of $(\varphi, \Gamma)$-modules; see \cite[\S II.1]{cherbonnier-colmez}
or \cite[Theorem~II.8]{berger-explicit}.
\end{remark}

We now treat the kernel of $\psi$. Although Theorem~\ref{T:psi kernel} is ultimately an easy consequence of previous results, its statement is in fact new.
\begin{defn}
Since $W_K$ is a quasi-Stein space, we may write it as the union of an ascending sequence $\{W_{K,n}\}$ of admissible affinoid subspaces. Given an affinoid space $X$ and a coherent sheaf $\calF$ on $X \times_K W_K$, define the \emph{module of boundary sections} of $\calF$ 
as 
\[
\Gamma^{\bd}(\calF) = 
\varinjlim_{n \to \infty} \Gamma(X \times_K (W_K \setminus W_{K,n}), \calF).
\]
\end{defn}

\begin{theorem} \label{T:psi kernel}
Suppose that $M \in \PhiGamma_{K,A}$,
$\tilde{M} \in \PhiGammatilde_{K,A}$ correspond as in
Theorem~\ref{T:tilde no tilde}.
Then for $X = \Maxspec(A)$,
we have canonical isomorphisms
\[
\Gamma^{\bd}(\calH^i_{\varphi,\Gamma}(\tilde{M} \boxtimes \Dfm_K)) \cong \begin{cases} M^{\psi=0} & i=1 \\ 0 & i \neq 1.
\end{cases}
\]
\end{theorem}
\begin{proof}
The vanishing for $i=0$ is apparent from Corollary~\ref{C:psi to deformation}
(because on a quasi-Stein space, a coherent sheaf is determined by its module of global sections); the vanishing for $i=2$ follows from the same considerations plus 
Proposition~\ref{P:psi finite}. For $i=1$, the morphism from the left side to the right side is induced by the map $\varphi-1$ in \eqref{eq:psi sequence}; to check that it is an isomorphism, we may reduce to the case where $A$ is reduced. In this case, we may
use Liu's extension of Tate's Euler characteristic formula \cite{liu-herr} (see also \cite[Theorem~2.3.11]{kpx}), applied pointwise on $X \times_K W_K$,
to see that the left side is a finite projective module over $A \widehat{\otimes}_K \Lambda_K$; we may then use \cite[Proposition~4.3.8]{kpx}, applied pointwise on $X$, to see that the right side is also a finite projective module of the same rank and that the map is an isomorphism.
\end{proof}

\section{Coda: beyond the cyclotomic tower}
\label{sec:coda}

To conclude, we put the previous discussion of the cyclotomic deformation into a context which we find suggestive for future work.

\begin{remark} \label{R:general field}
Let $L$ be any perfectoid field which is the completion of a Galois algebraic extension of $K$ with Galois group $G$.
By Lemma~\ref{L:Witt module descend}, we may also characterize $\PhiGammatilde_{K,A}$ 
as the category of objects of $\PhiMod_{L,A}$ equipped with continuous semilinear $G$-actions. For example, we may take $L = \CC_K$ to be a completed algebraic closure of $K$;
in this case, using Theorem~\ref{T:perfect equivalence} we get a description of
$\PhiGamma_{K,A}$ as objects of $\BPair_{\CC_K,A}$ equipped with continuous semilinear $G_K$-actions. In the case $A = \QQ_p$, this description is due to Berger \cite{berger-b-pairs}.
\end{remark}

\begin{remark}
In the language of \cite{kedlaya-liu1}, we may view objects of $\PhiGammatilde_{K,A}$
as sheaves on the pro-\'etale site of $K$ which are locally finite free modules over the ring sheaf $\CC_X \widehat{\otimes}_{\Qp} A$, equipped with an action of $\varphi$.
\end{remark}

\begin{remark}
For $M \in \PhiGamma_{K,A}$, we have stated descriptions of the objects $M^{\psi=1}$ 
(in Corollary~\ref{C:psi to deformation})
and $M^{\psi=0}$
(in Theorem~\ref{T:psi kernel}) of cyclotomic Iwasawa theory
in terms of the corresponding $\tilde{M} \in \PhiGammatilde_{K,A}$.
In light of the previous discussion, the object $\tilde{M}$ can be constructed, and computations can be made with it, without any direct reference to the cyclotomic extension of $\QQ_p$; the only appearance of the cyclotomic extension in the formulas is via the cyclotomic deformation $\Dfm_K$ on the weight space $W_K$.

Consequently, for a general $p$-adic Lie extension $L$ of $K$ with group $\Gamma$,
one may hope to get something meaningful by forming a suitable deformation space of representations $W_K$, using the homomorphism $G_K \to \Gal(L/K) \cong \Gamma$ to define an object of $\Rep_{W_K}(G_K)$, passing to the associated object in $\PhiGammatilde_{K,W_K}$,
taking the external tensor product with $\tilde{M}$,
and considering the cohomology of the result. For starters, in the case where $\Gamma$ is again a one-dimensional $p$-adic Lie group, it would be worth comparing this process to other constructions proposed as analogues of \cite{berger-explicit}, e.g.,
those of Berger--Fourquaux \cite{berger-fourquaux} and Schneider--Venjakob
\cite{schneider-venjakob}.
\end{remark}

\end{document}